\documentclass[12pt,a4paper]{amsart}
\usepackage{mathrsfs}
\usepackage{tikz}
\usepackage{amsfonts}
\usepackage{amsmath}
\usepackage{amssymb}
\usepackage{enumerate}
\usepackage{latexsym}
\usepackage{xcolor}
\usepackage{hyperref}
\usepackage{cite}
\usepackage{comment}
\usepackage[shortlabels]{enumitem}
\setlist[enumerate,1]{label={\upshape(\roman*)}}

\makeatletter

\newcommand{\Rmnum}[1]
{\expandafter\@slowromancap\romannumeral #1@}
\makeatother

\newtheorem{thm}{Theorem}[section]

\newtheorem{prop}[thm]{Proposition}
\newtheorem{lemma}[thm]{Lemma}

\newtheorem{example}[thm]{Example}
\newtheorem{defin}[thm]{Definition}

\theoremstyle{definition}
\newtheorem{remark}[thm]{Remark}

\title[Perfect codes in quartic Cayley graphs of generalized dihedral groups]{Perfect codes in quartic Cayley graphs on generalized dihedral groups}
\date{}

\thanks{*Corresponding author}

\author[DONG]{Chengcheng Dong}
\address{School of Science\\China University of Geosciences\\Beijing 100083\\China}
\email{dongcc@email.cugb.edu.cn}

\author[YANG]{Yuefeng Yang*}
\address{School of Science\\China University of Geosciences\\Beijing 100083\\China}
\email{yangyf@cugb.edu.cn}

\author[DONG]{Changchang Dong}
\address{School of Science\\China University of Geosciences\\Beijing 100083\\China}
\email{2024020015@cugb.edu.cn}

\begin{document}
	
\begin{abstract}
For a graph $\Gamma=(V\Gamma,E\Gamma)$, a subset $D$ of $V\Gamma$ is a perfect code in $\Gamma$ if every vertex of $\Gamma$ is dominated by exactly one vertex in $D$. In this paper, we classify all connected quartic Cayley graphs on generalized dihedral groups admitting a perfect code, and determine  all perfect codes in such graphs.
\end{abstract}
	
\keywords{Perfect code; Cayley graph; generalized dihedral group}
	
\subjclass[]{}
	
\maketitle
\section{Introduction}
A {\em graph} $\Gamma$ is a pair $(V\Gamma,E\Gamma)$ of vertex set $V\Gamma$ and edge set $E\Gamma$, where $E\Gamma$ is a subset of the set of 2-element subset of $V\Gamma$. In this paper, every group considered is finite, and every graph considered is finite, simple and undirected.

Let $\Gamma=(V\Gamma,E\Gamma)$ be a graph. Two distinct vertices $x,y\in V\Gamma$ are called {\em adjacent} if the set $\{x,y\}\in E\Gamma$. For a vertex $x\in V\Gamma$, let $N(x)$ denote the {\em neighborhood} of $x$ in $\Gamma$, which consists of all the vertices that are adjacent to $x$. A vertex $u$ {\em dominates} a vertex $v$ if either $u=v$ or $\{u,v\}\in E\Gamma$. A set $D\subseteq V\Gamma$ is called a
{\em dominating set} in $\Gamma$ if every vertex of $V\Gamma$ is dominated by a vertex in $D$. If $D$ is a dominating set in $\Gamma$ such
that every vertex of $\Gamma$ is dominated by exactly one vertex in $D$, then $D$ is called a {\em perfect code} in $\Gamma$. In some references, perfect codes are called {\em efficient dominating sets} \cite{DWB88,IJD03,DYP14,DYP17} or {\em independent perfect dominating sets} \cite{JL01,XM21}.

From the beginning of coding theory in the late 1940s, perfect codes have been important objects of study in information theory; see the surveys \cite{HO08,LJH75} for a large number of results on perfect codes. Biggs \cite{NB73} and Kratochv\'{i}l \cite{JK86} first investigated the problem of determining the existence of perfect codes in some families of graphs. Therefore, the notion of perfect codes can be generalized to graphs in a natural way. Also beginning with \cite{NB73}, perfect codes in distance transitive graphs have received considerable attention in the literature. See, for example, \cite{BE77,HP75,SDH80,KDS20}. It may safely be said that the most important class of distance transitive graphs, for coding
theory, is the class of Hamming graphs. The {\em Hamming graph} $H(n,q)$ is the Cartesian product
of $n$ copies of the complete graph of order $q$.

Since $H(n,q)$ is the Cayley graph of the additive group $\mathbb{Z}_q^n$ with respect to the set of all elements of $\mathbb{Z}_q^n$ with exactly one nonzero coordinate, perfect codes in Cayley graphs can be seen as another generalization of perfect codes in the classical setting. In general, for a group $G$ with the identity element $e$ and an {\em inverse-closed} subset $S$ of $G\setminus\{e\}$ (that is, $S^{-1}:=\{x^{-1}:x\in S\}=S$), the {\em Cayley graph} $\mathrm{Cay}(G,S)$ of $G$ with the {\em connection set} $S$ is defined to be the graph with the vertex set $G$ such that $x,y\in G$ are adjacent whenever $yx^{-1}\in S$. Clearly, $\mathrm{Cay}(G,S)$ is a $|S|$-regular graph, and $\mathrm{Cay}(G,S)$ is connected if and only if $G$ is generated by $S$, which is denoted by $G=\langle S\rangle$. Perfect codes in Cayley graphs have attracted considerable attention in recent years, see \cite{FR17,HH18} for short surveys and background information. Also in \cite{HH18}, Huang, Xia and Zhou introduced the concept of a subgroup perfect code of a group $G$, that is a subgroup of $G$ and a perfect code in some Cayley graph of $G$. Further results on subgroup perfect codes in Cayley graphs, see for examples \cite{CJ20,KY23, MX20,ZJ21,ZJ22}.
	
Perfect codes in Cayley graphs on abelian groups have been studied by several researchers in recent years. Obradovi\'{c}, Peters and Ru\v{z}i\'{c} \cite{NO07} gave a complete characterization of circulant graphs with two chord lengths that admit a perfect code, and classified cubic and quartic circulants (that is, Cayley graphs on cyclic groups) admitting a perfect code. Deng \cite{DYP14} gave a necessary and sufficient condition for a circulant graph to admit a perfect code with size a prime number, and characterized all perfect codes if exist. In \cite{FR17}, Feng, Huang and Zhou obtained a necessary and sufficient condition for a circulant graph of degree $p-1$ to admit a perfect code, where $p$ is an odd prime. And also in \cite{FR17}, authors gave a necessary and sufficient condition for a circulant graph of order $n$ and degree $p^l-1$ to have a perfect code, where $p$ is a prime and $p^l$ the largest power of $p$ dividing $n$. In \cite{CC19}, Caliskan, Miklavi$\mathrm{\check{c}}$ and $\mathrm{\ddot{O}}$zkan classified the cubic and quartic Cayley graphs on abelian groups admitting a perfect code. Kwon, Lee and Sohn \cite{KYS22} gave necessary and sufficient conditions for the existence of perfect codes in quintic circulants and classified these perfect codes. Yang, Ma and Zeng \cite{YYF} classified all connected quintic Cayley graphs on abelian groups that admit a perfect code, and determined completely all perfect codes in such graphs. Yu, Yang, Fan and Ma \cite{YSL24} classified strongly connected 2-valent Cayley digraphs on abelian groups admitting a perfect code, and determined completely all perfect codes of such digraphs. For more about perfect codes in Cayley graphs, see for examples \cite{IJD03,MM11,ZS19}.

One of the most natural next families of groups to consider are the generalized dihedral groups. Let $G$ be a group. For a subgroup $H\leqslant G$, we denote the {\em index} of $H$ in $G$ by $[G:H]$. Recall that $[G:H]=|G|/|H|$. For $a\in G$, let $o(a)$ denote the {\em order} of $a$, that is, the smallest positive integer $m$ such that $a^m=e$, where $e$ is the identity element. In particular, an element $a$ is called an {\em involution} if $o(a)=2$. A group $G$ is a {\em generalized dihedral group} generated by $A$ and $t$ if $A$ is an index $2$ abelian subgroup of $G$ and $t$ is an involution with $t\in G\setminus A$ such that $tat=a^{-1}$ holds for all $a\in A$. Caliskan, Miklavi$\mathrm{\check{c}}$, $\mathrm{\ddot{O}}$zkan and $\mathrm{\check{S}}$parl \cite{CC22} classified the connected cubic Cayley graphs on generalized dihedral groups which admit a perfect code. In this paper, we study connected quartic Cayley graphs on generalized dihedral groups that admit a perfect code, and characterize all perfect codes in such graphs.

In the remainder of this paper, $A$ always denotes a finite abelian group with an involution $t$, and $G$ denotes the generalized dihedral group generated by $A$ and $t$ where $tat=a^{-1}$ for all $a\in A$. For an integer $k$, let $\sigma(k)=(1-(-1)^k)/2$.

The following theorem classifies connected quartic Cayley graphs on generalized dihedral groups admitting a perfect code.

\begin{thm}\label{1.1} Let $\Gamma=\mathrm {Cay}(G,S)$ be a connected quartic Cayley graph on a generalized dihedral group. Without loss of generality assume $t\in S$. Then $\Gamma$ admits a perfect code if and only if one of the following holds:
\begin{enumerate}
\item\label{1.1-1} $S=\{s_1,s_1^{-1},t,ts_0\}$, $5\mid n$ and $s_1^m=s_0^{5u/2\pm m}$ for some $u\in\{0,2,4,\ldots,2(n/5-1)\}$;
\item\label{1.1-2} $S=\{t,ts_0,ts_1,ts_2\}$, $5\mid n$, $s_1^{m}=s_0^{5\alpha_1-am}$ and $s_2^{l}=s_0^{5\alpha_2-bl+aj}s_1^{j}$ for some $\alpha_1,\alpha_2\in\{0,1,\ldots,n/5-1\}$ and $j\in\{0,1,\ldots,m-1\}$, where $\{a,b\}=\{2+\sigma(v),v-\sigma(v)-1\}$ for $v\in\{2,3,4\}$.
\end{enumerate}
Here,  $s_0,s_1,s_2\in A$, $o(s_0)=n$, $m$ is the minimum positive integer such that $s_1^m\in\langle s_0\rangle$ and $l$ is the minimum positive integer such that $s_2^l\in \langle s_0,s_1\rangle$.
\end{thm}
%\begin{thm}\label{1.1}

The second main result determines all perfect codes in quartic Cayley graphs on generalized dihedral groups.
	
\begin{thm}\label{1.2}
With the notations in Theorem \ref{1.1}, suppose that $\Gamma$ admits a perfect code. Then one of the following holds:
\begin{enumerate}
\item\label{1.2-1} If Theorem \ref{1.1} \ref{1.1-1} holds, then all perfect codes containing $t$ are exactly
\begin{align}
\bigcup_{k=0}^{m-1}\bigcup_{u=0}^{2(n/5-1)}\{s_0^{5u/2+\epsilon k}s_1^{k}t,s_0^{3+5u/2+\epsilon k}s_1^{k}\}~\text{for}~\epsilon\in\{\pm1\};\nonumber
\end{align}
\item\label{1.2-2} If Theorem \ref{1.1} \ref{1.1-2} holds, then all perfect codes containing $t$ are exactly
$$\bigcup_{\alpha'=0}^{n/5-1}\bigcup_{j'=0}^{m-1}\bigcup_{k'=0}^{l-1}\left(\{s_0^{5\alpha'+aj'+bk'}s_1^{j'}s_2^{k'}t\}\cup\{s_0^{5\alpha'+aj'+bk'+v}s_1^{j'}s_2^{k'}\}\right).$$
\end{enumerate}
\end{thm}
%\begin{thm}\label{1.2}

In the remainder of this paper, to prove Theorems \ref{1.1} and \ref{1.2}, we always assume that $\Gamma=\mathrm {Cay}(G,S)$ is connected for some inverse closed subset $S\subseteq G$ with $|S|=4$. It follows that $\langle S\rangle =G$, and so $S$ contains at least one element from $tA$.

The paper is organized as follows. In Sections 2 and 3, under the assumption $|S\cap tA|\neq4$ and $|S\cap tA|=4$ respectively, we determine all graphs $\Gamma$ that admit a perfect code and determine all perfect codes in such graphs. In Section 4, we give the proofs of Theorems \ref{1.1} and \ref{1.2}.

\section{$|S\cap tA|\neq4$}

In this section, we consider the case $|S\cap tA|\neq4$. Since $\Gamma$ is connected, we have $|S\cap tA|\in\{1,2,3\}$. By \cite[Proposition 2.1]{VDT91}, one gets $5\mid |A|$.

\begin{lemma}\label{2.1}
If $|S\cap tA|=1$, then $\Gamma$ does not admit a perfect code.
\end{lemma}
\begin{proof}
Without loss of generality, we may assume $S=\{s,s',s_0,t\}$, where $s_0$ is an involution and $s,s',s_0\in A$. Since $5\mid |A|$, $S=S^{-1}$ and $\langle S\rangle=G$, $A=\langle s,s',s_0\rangle$ and $s$ is not an involution. It follows that $s'=s^{-1}$.

Since $S=\{s,s^{-1},s_0,t\}$, we get $\Gamma\simeq{\rm Cay}(A\times \langle t\rangle,\{(s,0),(s^{-1},0),(s_0,0),(0,t)\})$. Note that $A\times \langle t\rangle$ is an abelian group. If $s_0\in\langle s\rangle$, from \cite[Theorem 4.4]{CC19}, then $\Gamma$ does not admit a perfect code since $(0,t)\notin \langle(s,0)\rangle$; if $s_0\notin\langle s\rangle$, from \cite[Theorem 4.6]{CC19}, then $\Gamma$ does not admit a perfect code since $(0,t)\notin \langle(s,0),(s_0,0)\rangle$.
\end{proof}
%\begin{lemma}\label{2.1}
	
\begin{lemma}\label{2.2} If $|S\cap tA|=3$, then $\Gamma$ does not admit a perfect code.
\end{lemma}
\begin{proof} Without loss of generality, we may assume $S=\{s,t,ts_0,ts_1\}$, where $s,s_0,s_1\in A$. Since $\Gamma$ is undirected, $s$ is an involution. The fact $5\mid |A|$ implies $10\mid |A|$. Let $m$ be the minimum positive integer $i$ such that $s_1^i\in\langle s_0\rangle$. It follows that $1\leq m\leq o(s_1)$ and $s_1^m=s_0^h$ for some $h\in\{0,1,\ldots,o(s_0)-1\}$.

Assume the contrary, namely, $\Gamma$ admits a perfect code $D$. Since $\Gamma$ is vertex transitive, we may assume $t\in D$. Since $t,s_0t\in N(s_0)$, we have $s_0,s_0t\notin D$. The fact $N(s_0t)=\{s_0,s_0st,s_0^2,s_0s_1\}$ implies $s_0st,s_0^2$ or $s_0s_1\in D$.
	
Suppose $s_0st\in D$. Since $s_0st,s_0^2st,s_0^2\in N(s_0^2s)$ and $s_0st,s_0s_1\in N(s_0s_1s)$, one gets $s_0^2st,s_0^2,s_0^2s,s_0s_1,s_0s_1s\notin D$. Since $t\in D$ and $t,s_1t\in N(s_1)$, we have $s_1t\notin D$. Note that $s_0t\notin D$. The fact $N(s_0s_1)=\{s_0s_1s,s_0s_1t,s_1t,s_0t\}$ implies $s_0s_1t\in D$. Since $N(s_0^2s_1)=\{s_0^2s_1s,s_0^2s_1t,s_0s_1t,s_0^2t\}$, one has $s_0^2s_1,s_0^2s_1s,s_0^2t\notin D$. Since $N(s_0^2t)=\{s_0^2st,s_0^2,s_0^3,s_0^2s_1\}$, one gets $s_0^3\in D$. Since $s_0^3s\in N(s_0^3)$, we obtain $s_0^3s\notin D$, which implies $D\cap (N(s_0^2st)\cup\{s_0^2st\})=D\cap\{s_0^2t,s_0^2s,s_0^3s,s_0^2s_1s,s_0^2st\}=\emptyset$, a contradiction.
	
Suppose $s_0^2\in D$. Since $s_0^2,s_0st\in N(s_0^2s)$, $s_0^2,s_0^2s_1\in N(s_0^2t)$ and $s_0^2,s_0s_1\in N(s_0t)$, one gets $s_0st,s_0^2s,s_0^2s_1,s_0s_1\notin D$. Since $t\in D$ and $t,s_0s\in N(s_0)$, we have $s_0s\notin D$. Note that $s_0t\notin D$. The fact $N(s_0st)=\{s_0t,s_0s,s_0^2s,s_0s_1s\}$ implies $s_0s_1s\in D$. Since $s_0s_1s\in N(s_1st)$ and $s_0s_1s,s_0s_1t\in N(s_0s_1st)$, one gets $s_1st,s_0s_1t,s_0s_1st\notin D$. Since $t,s_1t\in N(s_1)$, we obtain $s_1t,s_1\notin D$. Since $N(s_1t)=\{s_1st,s_1,s_0s_1,s_1^2\}$ and $N(s_0s_1t)=\{s_0s_1st,s_0s_1,s_0^2s_1,s_0s_1^2\}$, we get $s_1^2,s_0s_1^2\in D$, which implies $s_1^2,s_0s_1^2\in D\cap N(s_1^2t)$, a contradiction.
	
Suppose $s_0s_1\in D$. Since $s_0s_1,s_0s_1st\in N(s_0s_1s)$, $s_0s_1,s_0^2s_1\in N(s_0s_1t)$ and $s_0s_1,s_0st\in N(s_0t)$, we have $s_0s_1st,s_0s_1s,s_0^2s_1,s_0s_1t,s_0st\notin D$. Since $t\in D$ and $t,s_0s\in N(s_0)$, one gets $s_0s\notin D$. Note that $s_0t\notin D$. The fact $N(s_0st)=\{s_0t,s_0s,s_0^2s,s_0s_1s\}$ implies $s_0^2s\in D$. Since $N(s_0^2st)=\{s_0^2t,s_0^2s,s_0^3s,s_0^2s_1s\}$, we obtain $s_0^2st,s_0^2t,s_0^2s_1s\notin D$. Since $N(s_0^2s_1s)=\{s_0^2s_1,s_0^2s_1st,s_0s_1st,s_0^2st\}$ and $N(s_0^2s_1)=\{s_0^2s_1s,s_0^2s_1t,s_0s_1t,s_0^2t\}$, we get $s_0^2s_1st,s_0^2s_1t\in D$. But $s_0^2s_1st$ and $s_0^2s_1t$ are adjacent in $\Gamma$, a contradiction.
	
This completes the proof of this lemma.
\end{proof}
%\begin{lemma}\label{2.2}

In the remainder of section, we only need to consider $|S\cap tA|=2$. Without loss of generality, we may assume $S=\{s_1,s_2,t,ts_0\}$, where $s_0,s_1,s_2\in A$.

\begin{lemma}\label{2.3}
If $s_1$ and $s_2$ are involutions, then $\Gamma$ does not admit a perfect code.
\end{lemma}
\begin{proof}
Since $s_1,s_2$ are involutions and $\langle t,ts_0\rangle=\langle s_0,t\rangle$, we may assume $s_1\notin\langle t,ts_0\rangle$. Note that ${\rm Cay}(\langle t, ts_0\rangle,\{t,ts_0\})\simeq {\rm Cay}(\mathbb{Z}_{2o(s_0)},\{\pm 1\})$. If $s_2\in\langle t,ts_0\rangle=\langle s_0,t\rangle$, then $s_2=s_0^{o(s_0)/2}$ with $2\mid o(s_0)$, and so $\Gamma\simeq{\rm Cay}(\mathbb{Z}_{2o(s_0)}\times\mathbb{Z}_{2},\{(\pm 1,0),(0,1),(o(s_0),0)\})$, which imply that $\Gamma$ does not admit a perfect code from \cite[Theorem 4.4]{CC19}. If $s_2\notin\langle t,ts_0\rangle$ and $s_2\in\langle t,ts_0,s_1\rangle$, then $s_2=s_0^{o(s_0)/2}s_1$ with $2\mid o(s_0)$, and so $\Gamma\simeq{\rm Cay}(\mathbb{Z}_{2o(s_0)}\times\mathbb{Z}_{2},\{(\pm 1,0),(0,1),(o(s_0),1)\})$, which imply that $\Gamma$ does not admit a perfect code from \cite[Theorem 4.5]{CC19}. If $s_2\notin\langle t,ts_0\rangle$ and $s_2\notin\langle t,ts_0,s_1\rangle$, then $\Gamma\simeq{\rm Cay}(\mathbb{Z}_{2o(s_0)}\times\mathbb{Z}_{2}\times\mathbb{Z}_{2},\{(\pm 1,0,0),(0,1,0),(0,0,1)\})$, which implies that $\Gamma$ does not admit a perfect code from \cite[Theorem 4.6]{CC19}. 
\end{proof}
%\begin{lemma}\label{2.3}

Next we consider the case that $s_1$ or $s_2$ is not an involution. Since $\Gamma$ is undirected, we get $s_2=s_1^{-1}$. Since $\langle S\rangle=G$, we have $A=\langle s_1,s_0\rangle$. Let $m$ be the minimum positive integer $i$ such that $s_1^i\in\langle s_0\rangle$. It follows that $1\leq m\leq o(s_1)$ and $s_1^m=s_0^h$ for some $h\in\{0,1,\ldots,o(s_0)-1\}$.

\begin{lemma}\label{2.4}
The graph $\Gamma$ is isomorphic to ${\rm Cay}(A',\{\pm\alpha,\pm\beta\})$, where $A'$ is an abelian group with $|A'|=2|A|$, $o(\alpha)=2o(s_0)$ and $o(\beta)=o(s_1)$. Moreover, if $\varphi$ is a mapping from $A'$ to $G$ such that
\begin{align}
\varphi(a\alpha+b\beta)=
\begin{cases}
	s_0^{a/2}s_1^b&{\rm if}~2\mid a\\
	s_0^{(a-1)/2}s_1^bt&{\rm if}~2\nmid a
\end{cases}\nonumber
\end{align}
for all $0\leq a\leq 2o(s_0)-1$ and $0\leq b\leq m-1$, then $\varphi$ is an isomorphism from ${\rm Cay}(A',\{\pm\alpha,\pm\beta\})$ to $\Gamma$.
\end{lemma}
\begin{proof}
Let $\Gamma_{2o(s_0),m,2(o(s_0)-h)}$ be the graph with vertex set $\mathbb{Z}_{2o(s_0)}\times\{0,1,\ldots,m-1\}$ whose edge set consists of $\{(a,b),(a+1,b)\}$, $\{(a,c),(a,c+1)\}$ and $\{(a,m-1),(a+2h,0)\}$ for $a\in\mathbb{Z}_{2o(s_0)}$, $b\in\{0,1,\ldots,m-1\}$ and $c\in\{0,1,\ldots,m-2\}$.
	
Let $\sigma$ be a mapping from $G$ to $\mathbb{Z}_{2o(s_0)}\times\{0,1,\ldots,m-1\}$ such that $\sigma(s_0^as_1^bt^{\epsilon})=(2a+\epsilon,b)$ for $0\leq a\leq o(s_0)-1$, $0\leq b\leq m-1$ and $0\leq\epsilon\leq1$. Suppose $\sigma(s_0^as_1^bt^{\epsilon})=\sigma(s_0^{a'}s_1^{b'}t^{\epsilon'})$ for $0\leq a,a'\leq o(s_0)-1$, $0\leq b,b'\leq m-1$ and $0\leq\epsilon,\epsilon'\leq1$. Since $2a+\epsilon=2a'+\epsilon'$, one has $\epsilon=\epsilon'$, and so $a=a'$. Hence, $b=b'$. Then $\sigma$ is injective. Let $0\leq i\leq 2o(s_0)-1$ and $0\leq j\leq m-1$. If $2\mid i$, then $\sigma(s_0^{i/2}s_1^j)=(i,j)$; if $2\nmid i$, then $\sigma(s_0^{(i-1)/2}s_1^jt)=(i,j)$. Then $\sigma$ is surjective. Thus, $\sigma$ is a bijection.
	
Note that $E\Gamma$ consists of $\{s_0^as_1^b,s_0^as_1^bt\},\{s_0^{a+1}s_1^b,s_0^as_1^bt\}$ and $\{s_0^{a}s_1^{b}t^{\epsilon},s_0^as_1^{b+1}t^{\epsilon}\}$ for $0\leq a\leq o(s_0)-1$, $0\leq b\leq m-1$ and $0\leq\epsilon\leq 1$. Since $s_1^m=s_0^h$, $\sigma$ is an isomorphism from $\Gamma$ to $\Gamma_{2o(s_0),m,2(o(s_0)-h)}$. By \cite[Proposition 2.3]{YYF}, the first statemet is valid. In view of the proof of \cite[Proposition 2.3]{YYF}, the second statement is also valid.
\end{proof}
%\begin{lemma}\label{2.4}

\begin{prop}\label{2.5}
The following hold:
\begin{enumerate}
	\item\label{2.5-1} $\Gamma$ admits a perfect code if and only if $5\mid o(s_0)$ and $h\equiv5u/2\pm m~({\rm mod}~o(s_0))$ for some $u\in\{0,2,4,\ldots,2(o(s_0)/5-1)\}$;
		
	\item\label{2.5-2} If $\Gamma$ admits a perfect code containing the element $t$, then $$\bigcup_{k=0}^{m-1}\bigcup_{u=0}^{2(o(s_0)/5-1)}\{s_0^{5u/2+\epsilon k}s_1^{k}t,s_0^{3+5u/2+\epsilon k}s_1^{k}\}$$ are all perfect code for $\epsilon\in\{\pm1\}$.
\end{enumerate}
\end{prop}
\begin{proof}
By Lemma \ref{2.4}, there exists an isomorphism $\varphi$ from ${\rm Cay}(A',\{\pm\alpha,\pm\beta\})$ to $\Gamma$ such that \begin{align}
\varphi(a\alpha+b\beta)=
\begin{cases}
	s_0^{a/2}s_1^b&{\rm if}~2\mid a\\
	s_0^{(a-1)/2}s_1^bt&{\rm if}~2\nmid a
\end{cases}\nonumber
\end{align}
for all $0\leq a\leq 2o(s_0)-1$ and $0\leq b\leq m-1$, where $A'$ is an abelian group with $|A'|=2|A|$, $o(\alpha)=2o(s_0)$ and $o(\beta)=o(s_1)$.
	
\ref{2.5-1} By Lemma \ref{2.4} and \cite[Theorem 4.11]{CC19}, $\Gamma$ admits a perfect code if and only if $5\mid o(\alpha)$ and $m\beta=(5u\pm 2m)\alpha$ for some $u\in\{0,1,\ldots,o(\alpha)/5-1\}$. Since $s_0^h=s_1^m=\varphi(m\beta)=\varphi((5u\pm 2m)\alpha)$, $\Gamma$ admits a perfect code if and only if $5\mid o(s_0)$ and $s_0^h=\varphi((5u\pm 2m)\alpha)=s_0^{5u/2\pm m}$ for some $u\in\{0,2,4,\ldots,2(o(s_0)/5-1)\}$. Then \ref{2.5-1} is valid.
	
\ref{2.5-2} By the proofs of \cite[Lemmas 4.7 and 4.10]{CC19}, $\bigcup_{k=0}^{m-1}\bigcup_{u=0}^{o(\alpha)/5-1}\{(5u+\epsilon 2k+1)\alpha+k\beta\}$ for some $\epsilon\in\{\pm1\}$ are all perfect codes of ${\rm Cay}(A',\{\pm\alpha,\pm\beta\})$ containing $\alpha$. Let $D$ be a perfect code in $\Gamma$ containing the element $t$. Since $\varphi(\alpha)=t$, we get $D=\bigcup_{k=0}^{m-1}\bigcup_{u=0}^{2(o(s_0)/5-1)}\{s_0^{5u/2+\epsilon k}s_1^{k}t,s_0^{3+5u/2+\epsilon k}s_1^{k}\}$ for some $\epsilon \in\{\pm1\}$. Then \ref{2.5-2} is also valid.
\end{proof}
%\begin{prop}\label{2.5}

\section{$|S\cap tA|=4$}

In this section, we consider the case that $|S\cap tA|=4$. Without loss of generality, we may assume $S=\{t,ts_0,ts_1,ts_2\}$, where $s_0,s_1,s_2\in A$. Since $\langle S\rangle=G$, one gets $A=\langle s_0,s_1,s_2\rangle$. Let $n$ be the order of $s_0$, and $m,l$ be the minimum positive integer such that $s_1^m\in \langle s_0\rangle$ and $s_2^l\in \langle s_0,s_1\rangle$ for $m\in \{1,2,\ldots,o(s_1)\}$ and $l\in\{1,2,\ldots,o(s_2)\}$. 

For an integer $k$, let $\sigma(k)=(1-(-1)^k)/2$. The following result constructs perfect codes in ${\rm Cay}(G,S)$ under the certain assumptions.

\begin{prop}\label{v-2/3/4}
Let $v\in\{2,3,4\}$ and $\{a,b\}=\{2+\sigma(v),v-\sigma(v)-1\}$. Suppose $5\mid n$. If $s_1^{m}=s_0^{5\alpha_1-am}$ and $s_2^{l}=s_0^{5\alpha_2-bl+aj}s_1^{j}$ with $0\leq \alpha_1,\alpha_2\leq n/{5}-1$ and $0\leq j\leq m-1$, then $$D:=\bigcup_{\alpha'=0}^{n/5-1}\bigcup_{j'=0}^{m-1}\bigcup_{k'=0}^{l-1}\left(\{s_0^{5\alpha'+aj'+bk'}s_1^{j'}s_2^{k'}t\}\cup\{s_0^{5\alpha'+aj'+bk'+v}s_1^{j'}s_2^{k'}\}\right)$$ is a perfect code of $\Gamma$.
\end{prop}
\begin{proof} Since the proof of the case $(a,b)=(2+\sigma(v),v-\sigma(v)-1)$ and the proof of the case $(a,b)=(v-\sigma(v)-1,2+\sigma(v))$ are similar, we may assume $(a,b)=(2+\sigma(v),v-\sigma(v)-1)$.
	
Since $A=\langle s_0,s_1,s_2\rangle$, we have $|D|=2nml/5=2|A|/5=|G|/5$. To show that $D$ is a perfect code in $\Gamma$, since $\Gamma$ has valency $4$, it suffices to prove that every vertex of $\Gamma$ is dominated by a vertex in $D$. Let $x\in V\Gamma$. Then $x=s_0^{p}s_1^{j'}s_2^{k'}t$ or $s_0^{p}s_1^{j'}s_2^{k'}$ with $0\leq p\leq n-1,0\leq j'\leq m-1$ and $0\leq k'\leq l-1$. It follows that there exists $r\in \{0,1,2,3,4\}$ such that $p=5\alpha'+aj'+bk'+r$ for some $\alpha'\in\{0,1,\ldots,n/5-1\}$.
	
\textbf{Case 1.} $x=s_0^{p}s_1^{j'}s_2^{k'}t$.
	
Note that $r\in\{0,1,2,3,4\}=\{0,v-1,v,v+a-5(\lceil v/2\rceil-1),v+b-5(\lfloor v/2\rfloor-1)\}$. If $r=0$, then $x=s_0^{5\alpha'+aj'+bk'}s_1^{j'}s_2^{k'}t\in D$. If $r\in\{v-1,v\}$, then $x=s_0^{5\alpha'+aj'+bk'+(v-1)}s_1^{j'}s_2^{k'}t$ or $s_0^{5\alpha'+aj'+bk'+v}s_1^{j'}s_2^{k'}t$, which implies that $x$ is dominated by $s_0^{5\alpha'+aj'+bk'+v}s_1^{j'}s_2^{k'}\in D$.

Suppose $r=v+a-5(\lceil v/2\rceil-1)$. Then $x=s_0^{5(\alpha'-\lceil v/2\rceil+1)+a(j'+1)+bk'+v}s_1^{j'}s_2^{k'}t$. If $0\leq j'\leq m-2$, then $x$ is dominated by $x\cdot ts_1=s_0^{5(\alpha'-\lceil v/2\rceil+1)+a(j'+1)+bk'+v}s_1^{j'+1}s_2^{k'}\in D.$ We only need to consider the case $j'=m-1$. Since $s_1^{m}=s_0^{5\alpha_1-am}$, $x$ is dominated by $$x\cdot ts_1=s_0^{5(\alpha'-\lceil v/2\rceil+1)+am+bk'+v}s_1^{m}s_2^{k'}=s_0^{5(\alpha'-\lceil v/2\rceil+1+\alpha_1)+bk'+v}s_2^{k'}\in D.$$
	
Suppose $r=v+b-5(\lfloor v/2\rfloor-1)$. Then $x=s_0^{5(\alpha'-\lfloor v/2\rfloor+1)+aj'+b(k'+1)+v}s_1^{j'}s_2^{k'}t$. If $0\leq k'\leq l-2$, then $x$ is dominated by $x\cdot ts_2=s_0^{5(\alpha'-\lfloor v/2\rfloor+1)+aj'+b(k'+1)+v}s_1^{j'}s_2^{k'+1}\in D.$ We only need to consider the case $k'=l-1$. Since $s_2^{l}=s_0^{5\alpha_2-bl+aj}s_1^{j}$, $x$ is dominated by $$x\cdot ts_2=s_0^{5(\alpha'-\lfloor v/2\rfloor+1)+aj'+bl+v}s_1^{j'}s_2^{l}=s_0^{5(\alpha'-\lfloor v/2\rfloor+1+\alpha_2)+a(j'+j)+v}s_1^{j'+j}.$$ If $j'+j\leq m-1$, then $x$ is dominated by $s_0^{5(\alpha'-\lfloor v/2\rfloor+1+\alpha_2)+a(j'+j)+v}s_1^{j'+j}\in D$. If $j'+j\geq m$, then $x$ is dominated by $$s_0^{5(\alpha'-\lfloor v/2\rfloor+1+\alpha_2)+a(j'+j)+v}s_1^{j'+j}=s_0^{5(\alpha'-\lfloor v/2\rfloor+1+\alpha_1+\alpha_2)+a(j'+j-m)+v}s_1^{j'+j-m}\in D$$ since $s_1^{m}=s_0^{5\alpha_1-am}$.

\textbf{Case 2.} $x=s_0^{p}s_1^{j'}s_2^{k'}$.
	
Note that $r\in \{0,1,5-a,5-b,v\}$. If $r=v$, then $x=s_0^{5\alpha'+aj'+bk'+v}s_1^{j'}s_2^{k'}\in D$. If $r\in\{0,1\}$, then $x=s_0^{5\alpha'+aj'+bk'}s_1^{j'}s_2^{k'}$ or $s_0^{5\alpha'+aj'+bk'+1}s_1^{j'}s_2^{k'}$, which implies that $x$ is dominated by $s_0^{5\alpha'+aj'+bk'}s_1^{j'}s_2^{k'}t\in D$.
	
Suppose $r=5-a$. Then $x=s_0^{5(\alpha'+1)+a(j'-1)+bk'}s_1^{j'}s_2^{k'}$. If $1\leq j'\leq m-1$, then $x$ is dominated by $x\cdot ts_1=s_0^{5(\alpha'+1)+a(j'-1)+bk'}s_1^{j'-1}s_2^{k'}t\in D.$ We only need to consider the case $j'=0$. Since $s_1^{m}=s_0^{5\alpha_1-am}$, $x$ is dominated by $$x\cdot ts_1=s_0^{5(\alpha'+1)-a+bk'}s_1^{-1}s_2^{k'}t=s_0^{5(\alpha'-\alpha_1+1)+a(m-1)+bk'}s_1^{m-1}s_2^{k'}t\in D.$$
	
Suppose $r=5-b$. Then $x=s_0^{5(\alpha'+1)+aj'+b(k'-1)}s_1^{j'}s_2^{k'}$. If $1\leq k'\leq l-1$, then $x$ is dominated by $x\cdot ts_2=s_0^{5(\alpha'+1)+aj'+b(k'-1)}s_1^{j'}s_2^{k'-1}t\in D.$ We only need to consider the case $k'=0$. Since $s_2^{l}=s_0^{5\alpha_2-bl+aj}s_1^{j}$, $x$ is dominated by $$x\cdot ts_2=s_0^{5(\alpha'+1)+aj'-b}s_1^{j'}s_2^{-1}t=s_0^{5(\alpha'-\alpha_2+1)+a(j'-j)+b(l-1)}s_1^{j'-j}s_2^{l-1}t.$$ If $j'-j\geq 0$, then $x$ is dominated by $s_0^{5(\alpha'-\alpha_2+1)+a(j'-j)+b(l-1)}s_1^{j'-j}s_2^{l-1}t\in D$. If $j'-j\leq -1$, then $x$ is dominated by $$s_0^{5(\alpha'-\alpha_2+1)+a(j'-j)+b(l-1)}s_1^{j'-j}s_2^{l-1}t=s_0^{5(\alpha'-\alpha_1-\alpha_2+1)+a(j'-j+m)+b(l-1)}s_1^{j'-j+m}s_2^{l-1}t\in D$$ since $s_1^{m}=s_0^{5\alpha_1-am}$.
\end{proof}

In the remainder of the section, we may assume that $\Gamma$ admits a perfect code $D$.

\begin{lemma}\label{ijk} Let $0\leq i\leq n-1,0\leq j\leq m-1$ and $0\leq k\leq l-1$. The following hold:
\begin{enumerate}
\item\label{ijk-1} $\left(\{s_0^{i+i'}s_1^{j}s_2^kt\mid0\leq i'\leq3\}\cup\{s_0^{i+j'}s_1^js_2^k\mid1\leq j'\leq 3\}\right)\cap D\neq\emptyset$;
\item\label{ijk-2} $\left(\{s_0^{i+i''}s_1^js_2^k\mid0\leq i''\leq 3\}\cup\{s_0^{i+j''}s_1^js_2^kt\mid0\leq j''\leq 2\}\right)\cap D\neq\emptyset$.
\end{enumerate}	
\end{lemma}
\begin{proof}
\ref{ijk-1} Since $\Gamma$ is vertex transitive, we may assume $i=j=k=0$. Assume the contrary, namely, 
\begin{align}\label{i'j'-1}
\left(\{s_0^{i'}t\mid0\leq i'\leq3\}\cup\{s_0^{j'}\mid1\leq j'\leq 3\}\right)\cap D=\emptyset.
\end{align}
The fact $N(s_0t)=\{s_0,s_0^{2},s_0s_1,s_0s_2\}$ implies $s_0s_1$ or $s_0s_2\in D$. Since $N(s_0^{2}t)=\{s_0^{2},s_0^{3},s_0^{2}s_1,s_0^{2}s_2\}$, we have $s_0^{2}s_1$ or $s_0^{2}s_2\in D$. Since $s_0s_1,s_0^{2}s_1\in N(s_0s_1t)$ and $s_0s_2,s_0^{2}s_2\in N(s_0s_2t)$, one has $s_0s_1,s_0^{2}s_2\in D$ or $s_0s_2,s_0^{2}s_1\in D$.

Since $N(s_0)=\{s_0t,t,s_0s_1^{-1}t,s_0s_2^{-1}t\}$, by \eqref{i'j'-1}, we get $s_0s_1^{-1}t$ or $s_0s_2^{-1}t\in D$. Since $N(s_0^{2})=\{s_0^{2}t,s_0t,s_0^{2}s_1^{-1}t,s_0^{2}s_2^{-1}t\}$, one gets $s_0^{2}s_1^{-1}t$ or $s_0^{2}s_2^{-1}t\in D$. Observe that $s_0s_1^{-1}t,s_0^{2}s_1^{-1}t\in N(s_0^{2}s_1^{-1})$ and $s_0s_2^{-1}t,s_0^{2}s_2^{-1}t\in N(s_0^{2}s_2^{-1})$. Then $s_0s_1^{-1}t,s_0^{2}s_2^{-1}t\in D$ or $s_0s_2^{-1}t,s_0^{2}s_1^{-1}t\in D$. Without loss of generality, we may assume $s_0s_1^{-1}t,s_0^{2}s_2^{-1}t\in D$.

Since $s_0^{2}s_2^{-1}t,s_0^{3}s_2^{-1}t\in N(s_0^{3}s_2^{-1})$, one obtains that $s_0^{3}s_2^{-1}t\notin D$. Since $N(s_0^3)=\{s_0^3t,s_0^2t,s_0^3s_1^{-1}t,s_0^3s_2^{-1}t\}$, from \eqref{i'j'-1}, we have $s_0^3s_1^{-1}t\in D$. It follows that $s_0^{3}s_1^{-1}\notin D$. The fact $s_0s_1^{-1}t,s_0^{2}s_1^{-1}t\in N(s_0^{2}s_1^{-1})$ implies $s_0^{2}s_1^{-1}t,s_0^{2}s_1^{-1}\notin D$. Note that $s_0^is_1,s_0^js_2\in D$, where $\{i,j\}=\{1,2\}$. Since $s_0^js_2,s_0^2s_1^{-1}s_2\in N(s_0^js_1^{-1}s_2t)$, one obtains $s_0^2s_1^{-1}s_2\notin D$, which implies $D\cap (N(s_0^{2}s_1^{-1}t)\cup \{s_0^{2}s_1^{-1}t\})=D\cap \{s_0^{2}s_1^{-1}t,s_0^{2}s_1^{-1},s_0^{3}s_1^{-1},s_0^{2},s_0^{2}s_1^{-1}s_2\}=\emptyset$ from \eqref{i'j'-1}, a contradiction.

\ref{ijk-2} The proof is similar, hence omitted.
\end{proof}
%\begin{lemma}\label{ijk}

\begin{lemma}\label{u=5} If $x\in D$, then $s_0^{v}x\notin D$ for $-4\leq v\leq 4$ and $v\neq0$.
\end{lemma}
\begin{proof} Since $\Gamma$ is vertex transitive, we may assume $x=t$. Assume the contrary, namely, there exists $v\in\{\pm1,\pm2,\pm3,\pm4\}$ such that $s_0^vt\in D$. Since the proof is similar, we only consider $1\leq v\leq 4$. Since $t,s_0t\in N(s_0)$, we obtain $s_0t\notin D$. It follows that $v\neq 1$.
	
\textbf{Case 1.} $v=2$.

Note that $s_0^2t\in D$. Since $s_0^2t,s_0t\in N(s_0^2)$ and $s_0\in N(t)$, one gets $s_0t,s_0^2,s_0\notin D$. The fact $N(s_0t)=\{s_0,s_0^2,s_0s_1,s_0s_2\}$ implies $s_0s_1$ or $s_0s_2\in D$. Without loss of generality, we may assume $s_0s_1\in D$. It follows that $s_0s_2\notin D$. Since $t,s_0^{-1}s_2t,s_2t\in N(s_2)$ and $s_0^2t,s_0s_2t,s_0^2s_2t\in N(s_0^2s_2)$, we have $s_0^{-1}s_2t,s_2,s_2t,s_0s_2t,s_0^2s_2,s_0^2s_2t\notin D$. Then $D\cap\{s_0^{-1}s_2t,s_2,s_2t,s_0s_2,s_0s_2t,s_0^2s_2,s_0^2s_2t\}=\emptyset$, contrary to Lemma \ref{ijk} \ref{ijk-1}.
	
\textbf{Case 2.} $v=3$.

Note that $s_0^3t\in D$. Since $t,s_0t\in N(s_0)$ and $s_0^3t,s_0^2t\in N(s_0^3)$, we get $s_0t,s_0,s_0^2t,s_0^3\notin D$. The fact $N(s_0t)=\{s_0,s_0^2,s_0s_1,s_0s_2\}$ implies $s_0^2,s_0s_1$ or $s_0s_2\in D$.
	
Suppose $s_0s_1,s_0s_2\notin D$. Then $s_0^2\in D$. The fact $s_0^2,s_0^2s_1,s_0^2s_2\in N(s_0^2t)$ implies $s_0^2s_1,s_0^2s_2\notin D$. Since $s_0^3t,s_0^2s_2t\in N(s_0^3s_2)$ and $s_0^3t,s_0^2s_1t\in N(s_0^3s_1)$, we get $s_0^2s_2t,s_0^2s_1t\notin D$. Since $t,s_0^{-1}s_2t,s_2t\in N(s_2)$ and $t,s_0^{-1}s_1t,s_1t\in N(s_1)$, we have $s_0^{-1}s_2t,s_2,s_2t,s_0^{-1}s_1t,s_1,s_1t\notin D$. By setting $(i,j,k)=(-1,0,1)$ and $(i,j,k)=(-1,1,0)$ in Lemma \ref{ijk} \ref{ijk-1}, we get $D\cap\{s_0^{-1}s_2t,s_2,s_2t,s_0s_2,s_0s_2t,s_0^2s_2,s_0^2s_2t\}\neq\emptyset$ and $D\cap\{s_0^{-1}s_1t,s_1,s_1t,s_0s_1,s_0s_1t,s_0^2s_1,s_0^2s_1t\}\neq\emptyset$. It follows that $s_0s_2t,s_0s_1t\in D$, and so $s_0s_2t,s_0s_1t\in D\cap N(s_0s_1s_2)$, a contradiction. Thus, $s_0s_1$ or $s_0s_2\in D$.
	
Without loss of generality, we may assume $s_0s_2\in D$. Observe that $s_0^2,s_0s_1\notin D$. Since $N(s_0s_2t)=\{s_0s_2,s_0^2s_2,s_0s_1s_2,s_0s_2^2\}$, we get $s_0s_2t,s_0^2s_2,s_0s_1s_2,s_0s_2^2\notin D$. Since $s_0s_2,s_1s_2\in N(s_2t)$, we have $s_1s_2\notin D$. The fact $t,s_1t\in N(s_1)$ implies $s_1t,s_1\notin D$. Note that $s_0^2t,s_0^3\notin D$. Since $N(s_1t)=\{s_1,s_0s_1,s_1^2,s_1s_2\}$ and $N(s_0^2t)=\{s_0^2,s_0^3,s_0^2s_1,s_0^2s_2\}$, one gets $s_1^2,s_0^2s_1\in D$. Since $s_1^2,s_1^2s_2\in N(s_1^2t)$ and $s_0^2s_1,s_0^2s_1s_2\in N(s_0^2s_1t)$, we obtain that $s_1^2s_2,s_0^2s_1s_2\notin D$. Since $s_0^3t\in D$ and $s_0^3t,s_0^2s_2t\in N(s_0^3s_2)$, we have $s_0^2s_2t,s_0^3s_2\notin D$. Since $N(s_0^2s_2t)=\{s_0^2s_2,s_0^3s_2,s_0^2s_1s_2,s_0^2s_2^2\}$ and $N(s_0^2s_2)=\{s_0^2s_2t,s_0s_2t,s_0^2s_1^{-1}s_2t,s_0^2t\}$, one obtains $s_0^2s_2^2,s_0^2s_1^{-1}s_2t\in D$. Since $s_0^2s_2^2\in N(s_0s_2^2t)$ and $s_0^2s_1^{-1}s_2t,s_0s_1^{-1}s_2^2t\in N(s_0^2s_1^{-1}s_2^2)$, one gets $s_0s_2^2t,s_0s_1^{-1}s_2^2t\notin D$. The fact $N(s_0s_2^2)=\{s_0s_2^2t,s_2^2t,s_0s_1^{-1}s_2^2t,s_0s_2t\}$ implies $s_2^2t\in D$. Since $s_2^2t,s_1s_2t\in N(s_1s_2^2)$, one obtains $s_1s_2t,s_1s_2^2\notin D$. Then $D\cap (N(s_1s_2t)\cup \{s_1s_2t\})=D\cap \{s_1s_2t,s_1s_2,s_0s_1s_2,s_1^2s_2,s_1s_2^2\}=\emptyset$, a contradiction.
		
\textbf{Case 3.} $v=4$.

Note that $s_0^4t\in D$ and $s_0^2t,s_0^3t\notin D$. Since $t,s_0t\in N(s_0)$ and $s_0^4t\in N(s_0^4)$, we get $s_0t,s_0,s_0^4\notin D$. By setting $(i,j,k)=(1,0,0)$ in Lemma \ref{ijk} \ref{ijk-2}, one has $D\cap\{s_0,s_0t,s_0^2,s_0^2t,s_0^3,s_0^3t,s_0^4\}\neq\emptyset$, and so $s_0^2$ or $s_0^3\in D$.
	
Suppose $s_0^2\in D$. Since $s_0^2,s_0^3\in N(s_0^2t)$, we have $s_0^3\notin D$. Since $N(s_0^3t)=\{s_0^3,s_0^4,s_0^3s_1,s_0^3s_2\}$, one gets $s_0^3s_1$ or $s_0^3s_2\in D$. Without loss of generality, we may assume $s_0^3s_2\in D$. Since $s_0^3s_2,s_0^3s_1^{-1}s_2,s_0^4s_1^{-1}s_2\in N(s_0^3s_1^{-1}s_2t)$ and $s_0^3s_2,s_0^2s_2\in N(s_0^2s_2t)$, one obtains $s_0^3s_1^{-1}s_2,s_0^3s_1^{-1}s_2t,s_0^4s_1^{-1}s_2,s_0^2s_2,s_0^2s_2t\notin D$. Since $t,s_0^{-1}s_2t,s_2t\in N(s_2)$, $s_0^2,s_0s_2\in N(s_0t)$ and $s_0^2,s_0^2s_1^{-1}s_2\in N(s_0^2s_1^{-1}t)$, we obtain $s_0^{-1}s_2t,s_2,s_2t,s_0s_2,s_0^2s_1^{-1}s_2\notin D$. By setting $(i,j,k)=(-1,0,1)$ in Lemma \ref{ijk} \ref{ijk-1}, one has $$D\cap\{s_0^{-1}s_2t,s_2,s_2t,s_0s_2,s_0s_2t,s_0^2s_2,s_0^2s_2t\}\neq\emptyset,$$ and so $s_0s_2t\in D$. Since $s_0s_2t,s_0^2s_1^{-1}s_2t\in N(s_0^2s_2)$ and $s_0s_2t,s_0s_1^{-1}s_2t\in N(s_0s_2)$, we get $s_0^2s_1^{-1}s_2t,s_0s_1^{-1}s_2t\notin D$. Since $s_0^4t,s_0^4s_1^{-1}s_2t\in N(s_0^4s_2)$, we get $s_0^4s_1^{-1}s_2t\notin D$, which implies $D\cap\{s_0s_1^{-1}s_2t,s_0^2s_1^{-1}s_2,s_0^2s_1^{-1}s_2t,s_0^3s_1^{-1}s_2,s_0^3s_1^{-1}s_2t,s_0^4s_1^{-1}s_2,s_0^4s_1^{-1}s_2t\}=\emptyset$, contrary to Lemma \ref{ijk} \ref{ijk-1}.

Suppose $s_0^2\notin D$. Then $s_0^3\in D$. Note that $s_0,s_0t\notin D$. The fact $N(s_0t)=\{s_0,s_0^2,s_0s_1,s_0s_2\}$ implies $s_0s_1$ or $s_0s_2\in D$. Without loss of generality, we may assume $s_0s_2\in D$. Since $s_0s_2,s_0^2s_2\in N(s_0s_2t)$ and $s_0s_2,s_0s_1^{-1}s_2,s_0^2s_1^{-1}s_2\in N(s_0s_1^{-1}s_2t)$, one obtains $s_0s_2t,s_0^2s_2,s_0s_1^{-1}s_2,s_0s_1^{-1}s_2t,s_0^2s_1^{-1}s_2\notin D$. Observe that $s_0^4t\in D$. Since $s_0^3,s_0^3s_2\in N(s_0^3t)$ and $s_0^4t,s_0^3s_2t,s_0^4s_2t\in N(s_0^4s_2)$, one gets $s_0^3s_2,s_0^3s_2t,s_0^4s_2,s_0^4s_2t\notin D$. By setting $(i,j,k)=(1,0,1)$ in Lemma \ref{ijk} \ref{ijk-1}, we get $$D\cap\{s_0s_2t,s_0^2s_2,s_0^2s_2t,s_0^3s_2,s_0^3s_2t,s_0^4s_2,s_0^4s_2t\}\neq\emptyset,$$ and so $s_0^2s_2t\in D$. Since $s_0^2s_2t,s_0^3s_1^{-1}s_2t\in N(s_0^3s_2)$ and $s_0^2s_2t,s_0^2s_1^{-1}s_2t\in N(s_0^2s_2)$, we have $s_0^3s_1^{-1}s_2t,s_0^2s_1^{-1}s_2t\notin D$. Since $s_0^3,s_0^3s_1^{-1}s_2\in N(s_0^3s_1^{-1}t)$ and $t,s_1^{-1}s_2t\in N(s_2)$, one obtains $s_0^3s_1^{-1}s_2,s_1^{-1}s_2t\notin D$, which implies $$D\cap\{s_1^{-1}s_2t,s_0s_1^{-1}s_2,s_0s_1^{-1}s_2t,s_0^2s_1^{-1}s_2,s_0^2s_1^{-1}s_2t,s_0^3s_1^{-1}s_2,s_0^3s_1^{-1}s_2t\}=\emptyset,$$ contrary to Lemma \ref{ijk} \ref{ijk-1}.
\end{proof}
%\begin{lemma}\label{u=5}

\begin{lemma}\label{v=234} If $x\in D\cap At$, then exactly one of $xts_0^2$, $xts_0^3$ and $xts_0^4$ belongs to $D$.
\end{lemma}
\begin{proof} Since $\Gamma$ is vertex transitive, we may assume $x=t$. By Lemma \ref{u=5}, we obtain $s_0t,s_0^2t,s_0^3t,s_0^4t\notin D$.  Since $s_0\in N(t)$, one gets $s_0\notin D$. By setting $(i,j,k)=(1,0,0)$ in Lemma \ref{ijk} \ref{ijk-1}, we get $D\cap\{s_0t,s_0^2,s_0^2t,s_0^3,s_0^3t,s_0^4,s_0^4t\}\neq \emptyset$. It follows that $\{s_0^2,s_0^3,s_0^4\}\cap D\neq\emptyset$.

Assume the contrary, namely, $|\{s_0^2,s_0^3,s_0^4\}\cap D|>1$. If $s_0^{i},s_0^{i+1}\in D$ for some $i\in\{2,3\}$, then $s_0^i,s_0^{i+1}\in N(s_0^it)\cap D$, a contradiction. It follows that $s_0^{2},s_0^{4}\in D$.

Since $s_0^2,s_0^3\in N(s_0^2t)$, we have $s_0^3\notin D$. Since $s_0^2,s_0^2s_1^{-1},s_0^3s_1^{-1}\in N(s_0^2s_1^{-1}t)$, we get $s_0^2s_1^{-1},s_0^2s_1^{-1}t,s_0^3s_1^{-1}\notin D$. Since $s_0^2,s_0^2s_2^{-1},s_0^3s_2^{-1}\in N(s_0^2s_2^{-1}t)$, we obtain $s_0^2s_2^{-1},s_0^2s_2^{-1}t,s_0^3s_2^{-1}\notin D$. Since $s_0^4,s_0^4s_1^{-1}\in N(s_0^4s_1^{-1}t)$ and $s_0^4,s_0^4s_2^{-1}\in N(s_0^4s_2^{-1}t)$, one gets $s_0^4s_1^{-1},s_0^4s_1^{-1}t,s_0^4s_2^{-1},s_0^4s_2^{-1}t\notin D$. The fact $t,s_0s_1^{-1}t,s_0s_2^{-1}t\in N(s_0)$ implies $s_0s_1^{-1}t,s_0s_2^{-1}t\notin D$. By setting $(i,j,k)=(1,-1,0)$ and $(i,j,k)=(1,0,-1)$ in Lemma \ref{ijk} \ref{ijk-1}, we obtain $$D\cap\{s_0s_1^{-1}t,s_0^2s_1^{-1},s_0^2s_1^{-1}t,s_0^3s_1^{-1},s_0^3s_1^{-1}t,s_0^4s_1^{-1},s_0^4s_1^{-1}t\}\neq\emptyset$$ and $$D\cap\{s_0s_2^{-1}t,s_0^2s_2^{-1},s_0^2s_2^{-1}t,s_0^3s_2^{-1},s_0^3s_2^{-1}t,s_0^4s_2^{-1},s_0^4s_2^{-1}t\}\neq\emptyset,$$ which imply $s_0^3s_1^{-1}t,s_0^3s_2^{-1}t\in D$. Then $s_0^3s_1^{-1}t,s_0^3s_2^{-1}t\in D\cap N(s_0^3)$, a contradiction.
\end{proof}
%\begin{lemma}\label{v=234}

For each $j,k\geq0$, let $L_{j,k}=\{s_0^is_1^js_2^kt,s_0^is_1^js_2^k:0\leq i\leq n-1\}$ and call it the {\em $(j,k)$-th layer} of $\Gamma$. 

\begin{lemma}\label{uv} Let $j,k\geq0$. If $x\in L_{j,k}\cap At\cap D$, then $5\mid n$ and $D\cap L_{j,k}=\cup_{\alpha=0}^{n/{5}-1}\{s_0^{5\alpha}x,xts_0^{v+5\alpha}\}$ for some $v\in\{2,3,4\}$.
\end{lemma}
\begin{proof} Note that, for each $(j,k)$-th layer,
$$(s_1^js_2^k,s_1^js_2^kt,s_0s_1^js_2^k,s_0s_1^js_2^kt,\ldots,s_0^{n-1}s_1^js_2^k,s_0^{n-1}s_1^js_2^kt)$$ is a cycle of length $2n$ in $\Gamma$. We refer to this cycle as $(j,k)$-cycle. Note that $x=s_0^is_1^js_2^kt\in D$ with $0\leq i\leq n-1$. By Lemma \ref{u=5}, any two vertices of $D\cap At^{\epsilon}$ belonging to the $(j,k)$-cycle are at least $10$ positions apart on this $(j,k)$-cycle for $\epsilon\in\{\pm1\}$. This shows that the intersection of $D\cap At^{\epsilon}$ with any $(j,k)$-th layer contains at most $2n/10$ vertices for $\epsilon\in\{\pm1\}$.

Since $|A|=|At|$ and $S\subseteq At$, we have $|A|=4|D\cap At|+|D\cap A|$ and $|At|=4|D\cap A|+|D\cap At|$, which imply $|D\cap A|=|D\cap At|$. Since $A=\langle s_0,s_1,s_2\rangle$, from the minimality of $m$ and $l$, we have $|D|=|D\cap A|+|D\cap At|=2|D\cap At^{\epsilon}|=2\sum_{0\leq j\leq m-1,0\leq k\leq l-1}|D\cap At^{\epsilon}\cap L_{j,k}|\leq2nml/5=|G|/5$ for $\epsilon\in\{\pm1\}$. Since $D=|G|/5$, the intersection of $D\cap At^{\epsilon}$ with each $(j,k)$-th layer contains exactly every tenth vertex of this cycle for $\epsilon\in\{\pm1\}$. Since $|L_{j,k}|=2n$, we get $5\mid n$ and $D\cap L_{j,k}=\cup_{\alpha=0}^{n/{5}-1}\{s_0^{5\alpha}x,xts_0^{v+5\alpha}\}$ for some $v\in\{2,3,4\}$ from Lemma \ref{v=234}.
\end{proof}
%\begin{lemma}\label{uv}

\begin{lemma}\label{jk}
Let $j,k\geq0$ and $\{a,b\}=\{2+\sigma(v),v-\sigma(v)-1\}$. Suppose $x\in At$ and $D\cap L_{j,k}=\cup_{\alpha'=0}^{n/{5}-1}\{s_0^{5\alpha'}x,xts_0^{v+5\alpha'}\}$ for some $v\in\{2,3,4\}$. Then $$D\cap L_{j+1,k}=\cup_{\alpha'=0}^{n/{5}-1}\{s_0^{5\alpha'}\cdot s_0^{a}s_1x,s_0^{a}s_1x\cdot ts_0^{v+5\alpha'}\},$$
$$D\cap L_{j,k+1}=\cup_{\alpha'=0}^{n/{5}-1}\{s_0^{5\alpha'}\cdot s_0^{b}s_2x,s_0^{b}s_2x\cdot ts_0^{v+5\alpha'}\}.$$
\end{lemma}
\begin{proof}
Since $\Gamma$ is vertex transitive, we may assume $x=t$. Then $j=k=0$.

\textbf{Case 1.} $v=2$.

Note that
\begin{align}\label{v2}
	D\cap L_{0,0}=\cup_{\alpha'=0}^{n/{5}-1}\{s_0^{5\alpha'}t,s_0^{2+5\alpha'}\}.
\end{align}
Since $N(s_0^3)=\{s_0^2t,s_0^3t,s_0^3s_1^{-1}t,s_0^3s_2^{-1}t\}$ and $N(s_0^4)=\{s_0^3t,s_0^4t,s_0^4s_1^{-1}t,s_0^4s_2^{-1}t\}$, by \eqref{v2}, we have $s_0^3s_1^{-1}t$ or $s_0^3s_2^{-1}t\in D$, and $s_0^4s_1^{-1}t$ or $s_0^4s_2^{-1}t\in D$. Since $s_0^3s_1^{-1}t,s_0^4s_1^{-1}t\in N(s_0^4s_1^{-1})$ and $s_0^3s_2^{-1}t,s_0^4s_2^{-1}t\in N(s_0^4s_2^{-1})$, we get $s_0^3s_x^{-1}t,s_0^4s_y^{-1}t\in D$ for $\{x,y\}=\{1,2\}$.

Since $N(s_0^3t)=\{s_0^3,s_0^4,s_0^3s_1,s_0^3s_2\}$ and $N(s_0^4t)=\{s_0^4,s_0^5,s_0^4s_1,s_0^4s_2\}$, from \eqref{v2}, one gets $s_0^3s_1$ or $s_0^3s_2\in D$, and $s_0^4s_1$ or $s_0^4s_2\in D$. Since $s_0^3s_1,s_0^4s_1\in N(s_0^3s_1t)$ and $s_0^3s_2,s_0^4s_2\in N(s_0^3s_2t)$, we have $s_0^3s_x,s_0^4s_y\in D$ for $\{x,y\}=\{1,2\}$.

Suppose $s_0^3s_x^{-1}t,s_0^4s_y^{-1}t,s_0^3s_x,s_0^4s_y\in D$ for $\{x,y\}=\{1,2\}$. \eqref{v2} implies $s_0^2\in D$. Since $s_0^2,s_0^3s_y^{-1}\in N(s_0^2s_y^{-1}t)$, one has $s_0^3s_y^{-1}\notin D$. Since $s_0^4s_y^{-1}t,s_0^3s_y^{-1}t\in N(s_0^4s_y^{-1})$, we get $s_0^3s_y^{-1}t,s_0^4s_y^{-1}\notin D$. Since $s_0^3s_x,s_0^3s_xs_y^{-1}\in N(s_0^3s_xs_y^{-1}t)$, one obtains $s_0^3s_xs_y^{-1}\notin D$. By \eqref{v2}, one has $D\cap (N(s_0^3s_y^{-1}t)\cup \{s_0^3s_y^{-1}t\})=D\cap \{s_0^3s_y^{-1}t,s_0^3s_y^{-1},s_0^4s_y^{-1},s_0^3s_xs_y^{-1},s_0^3\}=\emptyset$, a contradiction. Thus, $s_0^3s_x^{-1}t,s_0^4s_y^{-1}t,s_0^3s_y,s_0^4s_x\in D$ for $\{x,y\}=\{1,2\}$.

Since $s_0^4s_x,s_0^3s_x\in N(s_0^3s_xt)$, we get $s_0^3s_x,s_0^3s_xt\notin D$. Since $s_0^4s_y^{-1}t,s_0^3s_xs_y^{-1}t\in N(s_0^4s_xs_y^{-1})$, we have $s_0^3s_xs_y^{-1}t\notin D$. By \eqref{v2}, we obtain $s_0^3t\notin D$. The fact $N(s_0^3s_x)=\{s_0^3s_xt,s_0^2s_xt,s_0^3t,s_0^3s_xs_y^{-1}t\}$ implies $s_0^2s_xt\in D$. Since $s_0^2s_xt\in L_{y-1,x-1}\cap At\cap D$ and $s_0^4s_x\in D$, from Lemma \ref{uv}, one gets
\begin{align}
D\cap L_{y-1,x-1}=\cup_{\alpha'=0}^{n/5-1}\{s_0^{5\alpha'}\cdot s_0^{2}s_xt,s_0^{2}s_xt\cdot ts_0^{2+5\alpha'}\}.\label{x-2}
\end{align}

Since $s_0^3s_y,s_0^2s_y\in N(s_0^2s_yt)$, we obtain $s_0^2s_y,s_0^2s_yt\notin D$. Since $s_0^3s_x^{-1}t,s_0^2s_x^{-1}s_yt\in N(s_0^3s_x^{-1}s_y)$, one gets $s_0^2s_x^{-1}s_yt\notin D$. By \eqref{v2}, we have $s_0^2t\notin D$. The fact $N(s_0^2s_y)=\{s_0^2s_yt,s_0s_yt,s_0^2s_x^{-1}s_yt,s_0^2t\}$ implies $s_0s_yt\in D$. Since $s_0s_yt\in L_{x-1,y-1}\cap At\cap D$ and $s_0^3s_y\in D$, from Lemma \ref{uv}, one has
\begin{align}
D\cap L_{x-1,y-1}=\cup_{\alpha'=0}^{n/5-1}\{s_0^{5\alpha'}\cdot s_0s_yt,s_0s_yt\cdot ts_0^{2+5\alpha'}\}.\label{y-2}
\end{align}

Since $v=2$, we obtain $\{a,b\}=\{2,1\}$. Since $\{x,y\}=\{1,2\}$, from \eqref{x-2} and \eqref{y-2}, the desired result is valid.

\textbf{Case 2.} $v=3$.

Note that
\begin{align}\label{v3}
	D\cap L_{0,0}=\cup_{\alpha'=0}^{n/{5}-1}\{s_0^{5\alpha'}t,s_0^{3+5\alpha'}\}.
\end{align}
Since $N(s_0t)=\{s_0,s_0^2,s_0s_1,s_0s_2\}$ and $N(s_0^2)=\{s_0^2t,s_0t,s_0^2s_1^{-1}t,s_0^2s_2^{-1}t\}$, by \eqref{v3}, one has $s_0s_1$ or $s_0s_2\in D$, and $s_0^2s_1^{-1}t$ or $s_0^2s_2^{-1}t\in D$. Suppose $s_0s_x,s_0^2s_y^{-1}t\in D$ for $\{x,y\}= \{1,2\}$. Since $s_0s_x,s_0^2s_x\in N(s_0s_xt)$ and $s_0^2s_y^{-1}t,s_0^2s_xs_y^{-1}t\in N(s_0^2s_xs_y^{-1})$, one gets $s_0^2s_x,s_0s_xt,s_0^2s_xs_y^{-1}t\notin D$. Since $N(s_0^2s_x)=\{s_0^2s_xt,s_0s_xt,s_0^2t,s_0^2s_xs_y^{-1}t\}$, from \eqref{v3}, we obtain $s_0^2s_xt\in D$. Since $s_0^2s_xt,s_0^3s_xs_y^{-1}t\in N(s_0^3s_x)$, we have $s_0^3s_xs_y^{-1}t\notin D$. Since $t,s_xs_y^{-1}t\in N(s_x)$ and $s_0^3,s_0^3s_xs_y^{-1}\in N(s_0^3s_y^{-1}t)$, by \eqref{v3}, we get $s_xs_y^{-1}t,s_0^3s_xs_y^{-1}\notin D$. Since $s_0s_x,s_0s_xs_y^{-1},s_0^2s_xs_y^{-1}\in N(s_0s_xs_y^{-1}t)$, one has $s_0s_xs_y^{-1},s_0s_xs_y^{-1}t,s_0^2s_xs_y^{-1}\notin D$, which implies $$D\cap\{s_xs_y^{-1}t,s_0s_xs_y^{-1},s_0s_xs_y^{-1}t,s_0^2s_xs_y^{-1},s_0^2s_xs_y^{-1}t,s_0^3s_xs_y^{-1},s_0^3s_xs_y^{-1}t\}=\emptyset,$$ contrary to Lemma \ref{ijk} \ref{ijk-1}. Then $s_0s_x,s_0^2s_x^{-1}t\in D$ for $x\in \{1,2\}$.

Since $N(s_0^4)=\{s_0^4t,s_0^3t,s_0^4s_1^{-1}t,s_0^4s_2^{-1}t\}$ and $N(s_0^4t)=\{s_0^4,s_0^5,s_0^4s_1,s_0^4s_2\}$, from \eqref{v3}, we have $s_0^4s_1^{-1}t$ or $s_0^4s_2^{-1}t\in D$, and $s_0^4s_1$ or $s_0^4s_2\in D$. Suppose $s_0^4s_x^{-1}t,s_0^4s_y\in D$ for $\{x,y\}= \{1,2\}$. Since $s_0^4s_x^{-1}t,s_0^5s_x^{-1}t\in N(s_0^5s_x^{-1})$ and $s_0^4s_y,s_0^5s_x^{-1}s_y\in N(s_0^4s_x^{-1}s_yt)$, one gets $s_0^5s_x^{-1}t,s_0^5s_x^{-1},s_0^5s_x^{-1}s_y\notin D$. Since $N(s_0^5s_x^{-1}t)=\{s_0^5s_x^{-1},s_0^6s_x^{-1},s_0^5,s_0^5s_x^{-1}s_y\}$, from \eqref{v3}, one has $s_0^6s_x^{-1}\in D$. Since $s_0^6s_x^{-1},s_0^6s_x^{-1}s_y\in N(s_0^6s_x^{-1}t)$, we have $s_0^6s_x^{-1}s_y\notin D$. Since $s_0^3,s_0^3s_x^{-1}s_y\in N(s_0^3s_x^{-1}t)$ and $s_0^5t,s_0^5s_x^{-1}s_yt\in N(s_0^5s_y)$, by \eqref{v3}, one obtains $s_0^3s_x^{-1}s_y,s_0^5s_x^{-1}s_yt\notin D$. Since $s_0^4s_x^{-1}t,s_0^3s_x^{-1}s_yt,s_0^4s_x^{-1}s_yt\in N(s_0^4s_x^{-1}s_y)$, we obtain $s_0^3s_x^{-1}s_yt,s_0^4s_x^{-1}s_y,s_0^4s_x^{-1}s_yt\notin D$, which implies $$D\cap\{s_0^3s_x^{-1}s_y,s_0^3s_x^{-1}s_yt,s_0^4s_x^{-1}s_y,s_0^4s_x^{-1}s_yt,s_0^5s_x^{-1}s_y,s_0^5s_x^{-1}s_yt,s_0^6s_x^{-1}s_y\}$$ contrary to Lemma \ref{ijk} \ref{ijk-2}. Then $s_0^4s_x^{-1}t,s_0^4s_x\in D$ for $x\in \{1,2\}$.

Suppose $s_0s_x,s_0^2s_x^{-1}t,s_0^4s_x^{-1}t,s_0^4s_x\in D$ for $x\in\{1,2\}$. Since $s_0^3,s_0^3s_x^{-1}s_y\in N(s_0^3s_x^{-1}t)$, by \eqref{v3}, we get $s_0^3s_x^{-1}s_y\notin D$. The fact $s_0^2s_x^{-1}t,s_0s_x^{-1}s_yt,s_0^2s_x^{-1}s_yt\in N(s_0^2s_x^{-1}s_y)$ implies $s_0s_x^{-1}s_yt,s_0^2s_x^{-1}s_y,s_0^2s_x^{-1}s_yt\notin D$. Since $s_0^4s_x^{-1}t,s_0^3s_x^{-1}s_yt,s_0^4s_x^{-1}s_yt\in N(s_0^4s_x^{-1}s_y)$, we have $s_0^3s_x^{-1}s_yt,s_0^4s_x^{-1}s_y,s_0^4s_x^{-1}s_yt\notin D$, which implies $$D\cap\{s_0s_x^{-1}s_yt,s_0^2s_x^{-1}s_y,s_0^2s_x^{-1}s_yt,s_0^3s_x^{-1}s_y,s_0^3s_x^{-1}s_yt,s_0^4s_x^{-1}s_y,s_0^4s_x^{-1}s_yt\}=\emptyset,$$ contrary to Lemma \ref{ijk} \ref{ijk-1}. Thus, $s_0s_x,s_0^2s_x^{-1}t,s_0^4s_y^{-1}t,s_0^4s_y\in D$ for $\{x,y\}=\{1,2\}$.

Since $s_0^4s_y^{-1}t,s_0^4s_xs_y^{-1}t\in N(s_0^4s_xs_y^{-1})$, we have $s_0^4s_xs_y^{-1}t\notin D$. Since $s_0^4s_y,s_0^4s_x\in N(s_0^4t)$, we get $s_0^4s_x,s_0^4t\notin D$. Since $s_0^5t,s_0^4s_xt\in N(s_0^5s_x)$, by \eqref{v3}, one has $s_0^4s_xt\notin D$. The fact $N(s_0^4s_x)=\{s_0^4s_xt,s_0^3s_xt,s_0^4t,s_0^4s_xs_y^{-1}t\}$ implies $s_0^3s_xt\in D$. Since $s_0^3s_xt\in L_{y-1,x-1}\cap At\cap D$ and $s_0s_x\in D$, from Lemma \ref{uv}, one obtains
\begin{align}
D\cap L_{y-1,x-1}=\cup_{\alpha'=0}^{n/5-1}\{s_0^{5\alpha'}\cdot s_0^{3}s_xt,s_0^{3}s_xt\cdot ts_0^{3+5\alpha'}\}.\label{x-3}
\end{align}

Since $s_0s_x,s_0s_y\in N(s_0t)$ and $s_0^2s_x^{-1}t,s_0s_x^{-1}s_yt\in N(s_0^2s_x^{-1}s_y)$, we get $s_0s_y,s_0s_x^{-1}s_yt\notin D$. Since $t,s_yt\in N(s_y)$, from \eqref{v3}, we have $s_yt\notin D$. By \eqref{v3}, one has $s_0t\notin D$. The fact $N(s_0s_y)=\{s_0s_yt,s_yt,s_0s_x^{-1}s_yt,s_0t\}$ implies $s_0s_yt\in D$. Since $s_0s_yt\in L_{x-1,y-1}\cap At\cap D$ and $s_0^4s_y\in D$, by Lemma \ref{uv}, one gets
\begin{align}
D\cap L_{x-1,y-1}=\cup_{\alpha'=0}^{n/5-1}\{s_0^{5\alpha'}\cdot s_0s_yt,s_0s_yt\cdot ts_0^{3+5\alpha'}\}.\label{y-3}
\end{align}

Since $v=3$, we obtain $\{a,b\}=\{3,1\}$. Since $\{x,y\}=\{1,2\}$, from \eqref{x-3} and \eqref{y-3}, the desired result is valid.

\textbf{Case 3.} $v=4$.

Note that
\begin{align}\label{v4}
	D\cap L_{0,0}=\cup_{\alpha'=0}^{n/{5}-1}\{s_0^{5\alpha'}t,s_0^{4+5\alpha'}\}.
\end{align} Since $N(s_0t)=\{s_0,s_0^2,s_0s_1,s_0s_2\}$ and $N(s_0^2t)=\{s_0^2,s_0^3,s_0^2s_1,s_0^2s_2\}$, one gets $s_0s_1$ or $s_0s_2\in D$, and $s_0^2s_1$ or $s_0^2s_2\in D$. Since $s_0s_1,s_0^2s_1\in N(s_0s_1t)$ and $s_0s_2,s_0^2s_2\in N(s_0s_2t)$, we obtain $s_0s_x,s_0^2s_y\in D$ for $\{x,y\}=\{1,2\}$.

Since $N(s_0^2)=\{s_0^2t,s_0t,s_0^2s_1^{-1}t,s_0^2s_2^{-1}t\}$ and $N(s_0^3)=\{s_0^3t,s_0^2t,s_0^3s_1^{-1}t,s_0^3s_2^{-1}t\}$, from \eqref{v4}, we get $s_0^2s_1^{-1}t$ or $s_0^2s_2^{-1}t\in D$, and $s_0^3s_1^{-1}t$ or $s_0^3s_2^{-1}t\in D$. Since $s_0^2s_1^{-1}t,s_0^3s_1^{-1}t\in N(s_0^3s_1^{-1})$ and $s_0^2s_2^{-1}t,s_0^3s_2^{-1}t\in N(s_0^3s_2^{-1})$, one obtains $s_0^2s_x^{-1}t,s_0^3s_y^{-1}t\in D$ for $\{x,y\}=\{1,2\}$.

Suppose $s_0s_x,s_0^2s_y,s_0^2s_x^{-1}t,s_0^3s_y^{-1}t\in D$ for $\{x,y\}=\{1,2\}$. Since $s_0^2s_x^{-1}t,s_0^3s_x^{-1}t\in N(s_0^3s_x^{-1})$, we have $s_0^3s_x^{-1}t,s_0^3s_x^{-1}\notin D$. Since $s_0^2s_y,s_0^3s_x^{-1}s_y\in N(s_0^2s_x^{-1}s_yt)$, we obtain $s_0^3s_x^{-1}s_y\notin D$. Since $s_0^4,s_0^4s_x^{-1}\in N(s_0^4s_x^{-1}t)$, by \eqref{v4}, one gets $s_0^4s_x^{-1}\notin D$, which implies that $D\cap (N(s_0^3s_x^{-1}t)\cup \{s_0^3s_x^{-1}t\})=D\cap \{s_0^3s_x^{-1}t,s_0^3s_x^{-1},s_0^4s_x^{-1},s_0^3s_x^{-1}s_y,s_0^3\}=\emptyset$, a contradiction. Thus, $s_0s_x,s_0^2s_y,s_0^2s_y^{-1}t,s_0^3s_x^{-1}t\in D$ for $\{x,y\}=\{1,2\}$.

Since $s_0s_x,s_0^2s_x\in N(s_0s_xt)$, we have $s_0^2s_x,s_0s_xt\notin D$. The fact $s_0^2s_y^{-1}t,s_0^2s_xs_y^{-1}t\in N(s_0^2s_xs_y^{-1})$ implies $s_0^2s_xs_y^{-1}t\notin D$.  Since $N(s_0^2s_x)=\{s_0^2s_xt,s_0s_xt,s_0^2t,s_0^2s_xs_y^{-1}t\}$, by \eqref{v4}, we get $s_0^2s_xt\in D$. Since $s_0^2s_xt\in L_{y-1,x-1}\cap At\cap D$ and $s_0s_x\in D$, by Lemma \ref{uv}, one gets
\begin{align}
D\cap L_{y-1,x-1}=\cup_{\alpha'=0}^{n/5-1}\{s_0^{5\alpha'}\cdot s_0^{2}s_xt,s_0^{2}s_xt\cdot ts_0^{4+5\alpha'}\}.\label{x-4}
\end{align}

Since $s_0^2s_y,s_0^3s_y\in N(s_0^2s_yt)$, we obtain $s_0^3s_y,s_0^2s_yt\notin D$. Since $s_0^3s_x^{-1}t,s_0^3s_x^{-1}s_yt\in N(s_0^3s_x^{-1}s_y)$, one has $s_0^3s_x^{-1}s_yt\notin D$. By \eqref{v4}, one gets $s_0^3t\notin D$. The fact $N(s_0^3s_y)=\{s_0^3s_yt,s_0^2s_yt,s_0^3s_x^{-1}s_yt,s_0^3t\}$ implies $s_0^3s_yt\in D$. Since $s_0^3s_yt\in L_{x-1,y-1}\cap At\cap D$ and $s_0^2s_y\in D$, by Lemma \ref{uv}, we get
\begin{align}
D\cap L_{x-1,y-1}=\cup_{\alpha'=0}^{n/5-1}\{s_0^{5\alpha'}\cdot s_0^{3}s_yt,s_0^{3}s_yt\cdot ts_0^{4+5\alpha'}\}.\label{y-4}
\end{align}

Since $v=4$, we obtain $\{a,b\}=\{2,3\}$. Since $\{x,y\}=\{1,2\}$, from \eqref{x-4} and \eqref{y-4}, the desired result is valid.
\end{proof}
%\begin{lemma}\label{jk}

\begin{lemma}\label{234}
Let $v\in\{2,3,4\}$ and $\{a,b\}=\{2+\sigma(v),v-\sigma(v)-1\}$. Suppose $t\in D$. Then $5\mid n$, $s_1^{m}=s_0^{5\alpha_1-am}$, $s_2^{l}=s_0^{5\alpha_2-bl+aj}s_1^{j}$ and $$D=\bigcup_{\alpha'=0}^{n/5-1}\bigcup_{j'=0}^{m-1}\bigcup_{k'=0}^{l-1}\left(\{s_0^{5\alpha'+aj'+bk'}s_1^{j'}s_2^{k'}t\}\cup\{s_0^{5\alpha'+aj'+bk'+v}s_1^{j'}s_2^{k'}\}\right)$$ for some $\alpha_1,\alpha_2\in \{0,1,\ldots,n/{5}-1\}$ and $j\in \{0,1,\ldots,m-1\}$.
\end{lemma}
\begin{proof}
By Lemma \ref{uv}, we have $5\mid n$. Note that $D\cap L_{0,0}=\cup_{\alpha'=0}^{n/{5}-1}\{s_0^{5\alpha'}t,s_0^{v+5\alpha'}\}$ for some $v\in\{2,3,4\}$. By induction and Lemma \ref{jk}, we have 
\begin{align}
	D\cap L_{j',k'}=\cup_{\alpha'=0}^{n/5-1}\{s_0^{5\alpha'+aj'+bk'}s_1^{j'}s_2^{k'}t,s_0^{5\alpha'+aj'+bk'+v}s_1^{j'}s_2^{k'}\}\label{j'k'-1}
\end{align}
for $j',k'\geq0$.

Since $m$ is the minimum positive integer such that $s_1^m\in \langle s_0\rangle$ and $l$ is the minimum positive integer such that $s_2^l\in \langle s_0,s_1\rangle$, we get $D\cap L_{m,0}=D\cap L_{0,0}$ and $D\cap L_{0,l}=D\cap L_{j,0}$ for some $j\in\{0,1,\ldots,m-1\}$. By \eqref{j'k'-1}, there exist $\beta_1,\beta_2,\beta_3,\beta_4\in \{0,1,\ldots,n/5-1\}$ such that $s_0^{5\beta_1+am+v}s_1^{m}=s_0^{5\beta_2+v}$ and $s_0^{5\beta_3+bl+v}s_2^{l}=s_0^{5\beta_4+aj+v}s_1^{j}$. Then $s_1^m=s_0^{5(\beta_2-\beta_1)-am}$ and $s_2^l=s_0^{5(\beta_4-\beta_3)-bl+aj}s_1^j$. By \eqref{j'k'-1} again, the desired result follows.
\end{proof}
%\begin{lemma}\label{234}

\section{Proofs of Theorems \ref{1.1} and \ref{1.2}}

Now we are ready to give a proof of Theorem \ref{1.1}.

\begin{proof}[Proof of Theorem \ref{1.1}] By Proposition \ref{2.5} \ref{2.5-1} and Proposition \ref{v-2/3/4}, the sufficiency is valid.
	
Now we prove the necessity. By Lemmas \ref{2.1} and \ref{2.2}, we get $|S\cap At|=2$ or $4$.

Suppose $|S\cap At|=2$. By Lemma \ref{2.3}, we may assume $S=\{s_1,s_1^{-1}, t, ts_0\}$ with $s_0,s_1\in A$. Let $o(s_0)=n$ and $m$ be the minimum positive integer such that $s_1^m\in\langle s_0\rangle$. Since $s_1^m=s_0^h$ for some $h\in\{0,1,\ldots,n-1\}$, from Proposition \ref{2.5} \ref{2.5-1}, we get $5\mid n$ and $h\equiv5u/2\pm m~({\rm mod}~n)$ for some $u\in\{0,2,4,\ldots,2(n/5-1)\}$. \ref{1.1-1} is valid.

Suppose $|S\cap At|=4$. Let $S=\{t,ts_0,ts_1,ts_2\}$ and $o(s_0)=n$ with $s_0,s_1,s_2\in A$. Since $\Gamma$ is vertex transitive, we may assume $t\in D$. Let $m$ be the minimum positive integer such that $s_1^m\in\langle s_0\rangle$ and $l$ be the minimum positive integer such that $s_2^l\in \langle s_0,s_1\rangle$. By Lemma \ref{234}, \ref{1.1-2} is valid.
\end{proof}

Next, we give a proof of the Theorem \ref{1.2}.

\begin{proof}[Proof of Theorem \ref{1.2}] By Proposition \ref{2.5} \ref{2.5-2}, Proposition \ref{v-2/3/4} and Lemma \ref{234}, the desired result is valid.
\end{proof}
 	
\section*{Acknowledgements}
Yuefeng Yang is supported by NSFC (12101575, 52377162),  Changchang Dong is supported by the Fundamental Research Funds for the Central Universities.

\section*{Data Availability Statement}
	
No data was used for the research described in the article.


\begin{thebibliography}{00}
\bibitem{DWB88} D.W. Bange, A.E. Barkauskas and P.J. Slater, Efficient dominating sets in graphs, \textit{Appl. Discrete Math.}, R.D. Ringeisen and F.S. Roberts (Ed(s)), (Philadelphia, SIAM, 1988), 189--199.

\bibitem{NB73} N. Biggs, Perfect codes in graphs, \textit{J. Combin. Theory Ser. B}, 15 (1973), 289--296.

\bibitem{BE77} E. Bannai, On perfect codes in the Hamming schemes $H(n,q)$ with $q$ arbitrary, \textit{J. Combin.
Theory Ser. A}, 23 (1977), 52--67.

\bibitem{CC19} C. Caliskan, $\mathrm{\check{S}}$. Miklavi$\mathrm{\check{c}}$ and S. $\mathrm{\ddot{O}}$zkan, Domination and efficient domination in cubic and quartic Cayley graphs on abelian groups, \textit{Discrete Appl. Math.}, 271 (2019), 15--24.
		
\bibitem{CC22} C. Caliskan, $\mathrm{\check{S}}$. Miklavi$\mathrm{\check{c}}$, S. $\mathrm{\ddot{O}}$zkan, and P. $\mathrm{\check{S}}$parl, Efficient domination in Cayley graphs of generalized dihedral groups, \textit{Discussiones Math. Graph Theory}, 42 (2022), 823--841.

\bibitem{CJ20} J. Chen, Y. Wang and B. Xia, Characterization of subgroup perfect codes in Cayley
graphs, \textit{Discrete Math.}, 343 (2020), 111813.
		
\bibitem{IJD03} I.J. Dejter and O. Serra, Efficient dominating sets in Cayley graphs, \textit{Discrete Appl. Math.}, 129 (2003), 319--328.
		
\bibitem{DYP14} Y. Deng, Efficient dominating sets in circulant graphs with domination number prime, \textit{Inform. Process. Lett.}, 114 (2014), 700--702.

\bibitem{DYP17} Y. Deng, Y. Sun, Q. Liu and H. Wang, Efficient dominating sets in circulant graphs, \textit{Discrete Math.}, 340 (2017), 1503--1507.

\bibitem{FR17} R. Feng, H. Huang and S. Zhou, Perfect codes in circulant graphs, \textit{Discrete Math.}, 340 (2017), 1522--1527.

\bibitem{HP75} P. Hammond and D. H. Smith, Perfect codes in the graphs $O_k$, \textit{J. Combin. Theory Ser. B}, 19
(1975), 239--255.

\bibitem{HO08} O. Heden, A survey of perfect codes, \textit{Adv. Math. Commun.}, 2 (2008), 223--247.

\bibitem{HH18} H. Huang, B. Xia and S. Zhou, Perfect codes in Cayley graphs, \textit{SIAM J. Discrete Math.}, 32 (2018), 548--559.

\bibitem{KY23} Y. Khaefi, Z. Akhlaghi and B. Khosravi, On the subgroup perfect codes in Cayley graphs, \textit{Des. Codes Cryptogr.}, 91 (2023), 55--61.
		
\bibitem{JK86} J. Kratochv\'{i}l, Perfect codes over graphs, \textit{J. Combin. Theory Ser. B}, 40 (1986), 224--228.

\bibitem{KDS20} D.S. Krotov, The existence of perfect codes in Doob graphs, \textit{IEEE Trans. Inform. Theory}, 66 (2020), 1423--1427.

\bibitem{KYS22} Y.S. Kwon, J. Lee and M.Y. Sohn, Classification of efficient dominating sets of circulant graphs of degree 5, \textit{Graph Combin.}, 38 (2022), 120.
		
\bibitem{JL01} J. Lee, Independent perfect domination sets in Cayley graphs, \textit{J. Graph Theory}, 37 (2001), 213--219.

\bibitem{LJH75} J.H. van Lint, A survey of perfect codes, \textit{Rocky Mountain J. Math.}, 5 (1975), 199--224.

\bibitem{MX20} X. Ma, G.L. Walls, K. Wang and S. Zhou, Subgroup perfect codes in Cayley graphs, \textit{SIAM J. Discrete Math.}, 34 (2020), 1909--1921.

\bibitem{MM11} M. Mollard, On perfect codes in Cartesian products of graphs, \textit{European J. Combin.}, 32 (2011), 398--403.
		
\bibitem{NO07} N. Obradovi\'{c}, J. Peters and G. Ru\v{z}i\'{c}, Efficient domination in circulant graphs with two chord lengths, \textit{Inf. Process. Lett.}, 102 (2007), 253--258.

\bibitem{SDH80} D.H. Smith, Perfect codes in the graphs $O_k$ and $L(O_k)$, \textit{Glasgow Math. J.}, 21 (1980),
169--172.

\bibitem{VDT91} D.T. Vuza, Supplementary sets and regular complementary unending canons (part one), \textit{Perspect. New Music}, 29 (1991), 22--49.

\bibitem{XM21} X. Wang, S. Xu and X. Li, Independent perfect dominating sets in semi-Cayley graphs, \textit{Theor. Comput. Sci.}, 864 (2021), 50--57.

\bibitem{YYF} Y. Yang, X. Ma and Q. Zeng, Perfect codes in quintic Cayley graphs on Abelian groups, \textit{arXiv:} 2207.06743.

\bibitem{YSL24} S. Yu, Y. Yang, Y. Fan and X. Ma, Perfect codes in 2-valent Cayley digraphs on abelian groups, \textit{Discrete Appl. Math.}, 357 (2024), 236--240.

\bibitem{ZJ21} J. Zhang and S. Zhou, On subgroup perfect codes in Cayley graphs. \textit{European J. Combin.}, 91 (2021), 103228.

\bibitem{ZJ22} J. Zhang and S. Zhou, Corrigendum to ``On subgroup perfect codes in Cayley graphs'' [European J. Combin. 91 (2021) 103228], \textit{European J. Combin.}, 101 (2022), 103461.

\bibitem{ZS19} S. Zhou, Cyclotomic graphs and perfect codes, \textit{J. Pure Appl. Algebra}, 223 (2019), 931--947.
		
\end{thebibliography}
\end{document}